\documentclass[12pt,a4paper,oneside,reqno]{amsart}
\title[]{Hopf algebroids and twists \\~ \\ for quantum projective spaces}
\author{Ludwik Dabrowski, Giovanni Landi, Jacopo Zanchettin}
\address[LD]{\textit{SISSA, Via Bonomea 265, 34136 Trieste, Italy}}
\email{dabrow@sissa.it}
\address[GL]{\textit{Universit\`a di Trieste, Via A. Valerio 12/1, 34127 Trieste, Italy}
	\newline \indent \textit{Institute for Geometry and Physics (IGAP) Trieste, Italy} \newline \indent \textit{and INFN, Sezione di Trieste, Trieste, Italy}.}
\email{landi@units.it}
\address{\textit{SISSA, Via Bonomea 265, 34136 Trieste, Italy}}
\email{jzanchet@sissa.it}
\date{}
\usepackage{amsmath,amssymb,amstext,amsthm,inputenc,graphicx,charter,authblk,tensor,tikz-cd}
\usepackage[top=2.5cm,bottom=2.5cm,left=2.5cm,right=2.5cm,heightrounded,bindingoffset=10mm]{geometry}
\usepackage{color}
\newcommand{\rosso}[1]{{\color{red}{\sl #1}}}

\newcommand{\sr}[1]{\rosso{$\overline{\rule{0pt}{1pt}\smash[t]{\text{\color{black}#1}}}$}}
\usepackage{hyperref}
\hypersetup{colorlinks,
	linkcolor=black!75!red,
	citecolor=black,
	pdftitle={},
	pdfproducer={pdfLaTeX},
	pdfpagemode=None,
	bookmarksopen=true
	bookmarksnumbered=true}
\usepackage{slashed}
\numberwithin{equation}{section}
\theoremstyle{plain} 
\newtheorem{thm}{Theorem}[section] 
\newtheorem{cor}[thm]{Corollary} 
\newtheorem{lem}[thm]{Lemma} 
\newtheorem{prop}[thm]{Proposition} 
\newtheorem{defn}[thm]{Definition}
\newtheorem{expl}[thm]{Example}
\newtheorem{remark}[thm]{Remark}
\newcommand{\zero}[1]{{#1}_{\scriptscriptstyle{(0)}}}
\newcommand{\one}[1]{{#1}_{\scriptscriptstyle{(1)}}}

\newcommand{\beq}{\begin{equation}}
	\newcommand{\eeq}{\end{equation}}
\newcommand{\beqn}{\begin{equation*}}
	\newcommand{\eeqn}{\end{equation*}}
\newcommand{\tuno}[1]{{#1}{}{}^{\scriptscriptstyle{\langle 1 \rangle}}}
\newcommand{\tdue}[1]{{#1}{}{}^{\scriptscriptstyle{\langle 2 \rangle}}}

\begin{document}
	\maketitle
	
	\begin{abstract}
		We study the relationship between antipodes on a Hopf algebroid \sr{$\mathcal{H}$} in the sense of B\"ohm--Szlachanyi and the group of twists that lies inside the associated convolution algebra. We specialize to the case of a faithfully flat $H$-Hopf--Galois extensions $B\subseteq A$ and related Ehresmann--Schauenburg bialgebroid. In particular, we find that the twists are in one-to-one correspondence with $H$-comodule algebra automorphism of $A$. We work out in detail the $U(1)$-extension ${\mathcal O}(\mathbb{C}P^{n-1}_q)\subseteq {\mathcal O}(S^{2n-1}_q)$ on the quantum projective space and show how to get an antipode on the bialgebroid out of the $K$-theory of the base algebra ${\mathcal O}(\mathbb{C}P^{n-1}_q)$. 
	\end{abstract}
	
	\tableofcontents
	
	\parskip =.75 ex

	\section*{Introduction}
	The relationship between Hopf algebroids and groupoids is similar in spirit to the duality between Hopf algebras and groups.
	At the algebraic level, a Hopf algebroid is, roughly, a Hopf algebra over a (possibly) noncommutative algebra rather than over a field.
	While the algebra and coalgebra structures have been generalized in this direction and are well-understood by now \cite{brz-wis,tak_gpalg},
	for a complete notion of a Hopf algebroid there is not a wide consensus yet.
	The major issue regards the antipode map which, for a Hopf algebra $H$ is a bialgebra map $S:H\to H^{copop}$.
	In contrast, for a bialgebroid $\mathcal{H}$, the space $\mathcal{H}^{copop}$ is not a bialgebroid in the same sense 
	as $\mathcal{H}$.
	A first definition of an antipode in this setting appeared in \cite{lu1996hopf} and later in \cite{bohm2004hopf,bohm2019alternative},
	while an abstract category-theoretical approach that does not refer to an antipode can be found in \cite{day-street-quantum,schau_dd}. The latter definition turns out to be weaker than the former since a Hopf algebroid with an (invertible) antipode in the sense of \cite{bohm2004hopf,bohm2019alternative} is also a Hopf algebroid in the sense of \cite{schau_dd} (see Proposition~\ref{prop:full_is_Hopf} below). The other way around is not true in general: there are bialgebroids that satisfy the requirement of \cite{schau_dd} but they do not admit any antipode \cite{rovi2014hopf}. In \cite{bohm2019alternative} the notion of twist of an antipode is introduced. Twists are elements in  
	in the convolution algebra of the bialgebroid $\mathcal{H}$  that are invertible, act trivially on $1_{\mathcal{H}}$ and their action is multiplicative. Theorem 4.3 in \cite{bohm2019alternative} relates twists with antipodes on a bialgebroid, when the latter exist. 
	
	In the present work, we address the problem of the existence of an antipode for a special case, that is the Ehresemann-Schauenburg (E-S) bialgebroid associated with a Hopf--Galois extension $B\subseteq A$, which is a noncommutative version of the gauge groupoid associated with a principal $G$-bundle $P\to X$ \cite{mackenzie1987lie}. The latter is the Lie groupoid with total space the quotient $(P\times P)/G$ of the diagonal action and base space $X$. 
	Our Definition~\ref{defn:twists}  of a twist does not assume the condition in the definition 4.1 of \cite{bohm2019alternative}, a condition involving the antipode. Thus twists can be given also for bialgebroids. 
	Then on a general level, in Theorem~\ref{thm:Bohm_thm} we prove that once there is an antipode, the collection of all antipodes is in one-to-one correspondence with the group of twists as in \cite{bohm2019alternative}. 
	For the E-S bialgebroid we characterize the group of twists showing that they are in bijective  correspondence with unital comodule algebra automorphism of the total space algebra $A$. As an example we study in details the  $U(1)$-Hopf--Galois extension ${\mathcal O}(\mathbb{C}P^{n-1}_q)\subseteq {\mathcal O}(S^{2n-1}_q)$ on the quantum projective space $\mathbb{C}P^{n-1}$ and in addition we show how to construct an antipode starting from the $K$-theory of the base space algebra ${\mathcal O}(\mathbb{C}P^{n-1}_q)$.
	
	The outline of the paper is as follows. In the first section, we recall the basic algebraic concepts, and in the second we do an overview of the various definitions of Hopf algebroid with some basic proofs about  
	the convolution algebra of a bialgebroid. In the third section we state and prove that the $\mathrm{flip}$ is always an antipode whenever the structure Hopf algebra of a Hopf--Galois extension is commutative (this is a known result, but there is no explicit proof in the literature); moreover Proposition~\ref{prop:group_iso} relates the twists with comodule algebra automorphisms. Finally in the fourth section we study in details the $U(1)$-extension ${\mathcal O}(\mathbb{C}P^{n-1}_q)\subseteq {\mathcal O}(S^{2n-1}_q)$ on the quantum projective space.

	\section{Algebraic preliminaries}
	In this section we recall basic algebraic concepts of comodule algebra extensions 
	and conditions for them to be 
	Hopf--Galois (we refer to \cite{Montgomery1993HopfAA} for detais).  In the following $\mathbb{K}$ denotes a field, and we use the convention 
	$\otimes=\otimes_{\mathbb{K}}$. An algebra is denoted by $(A,m_A,\nu_A)$ where $m_A:A\otimes A \to A$ is the product and $\nu_A:\mathbb{K} \to A$ is the unit; $(C,\Delta_C,\epsilon_C)$ denotes a coalgebra with coproduct $\Delta_C:C \to C\otimes C$ for which we adopt the Sweedler notation $\Delta_C(c)=c_{(1)}\otimes c_{(2)}$ and counit $\epsilon_C:C\to   \mathbb{K}$. When it is clear from the context, we drop the subscript of the maps. When we write vector space we mean $\mathbb{K}$-vector space. When we write algebra (coalgebra) we mean a unital (counital) algebra (coalgebra) over $\mathbb{K}$.

	\subsection{Comodule algebras}
	Recall that a bialgebra $H$ is the datum of $(H,m,\nu,\Delta,\epsilon)$ where $(H,m,\nu)$ is an algebra and $(H,\Delta,\epsilon)$ is a coalgebra such that coproduct and counit are algebra morphisms. A bialgebra is a Hopf algebra when there exists 
	a map $S : H \to H$, the antipode, 
	which is an anti-algebra and anti-coalgebra morphism and satisfies the condition
	\begin{equation*}
		m\circ(S\otimes\mathrm{id})\circ\Delta=\nu\circ \epsilon=m\circ(\mathrm{id}\otimes S)\circ\Delta,
	\end{equation*}
	or explicitly $S(h_{(1)})h_{(2)}=\epsilon(h)1_H=h_{(1)}S(h_{(2)})$.
	
	A \textbf{right $H$-comodule} ($V, \rho$) is a vector space $V$ with a \textbf{right coaction} of $H$, that is a 
	linear map $\rho:V\to V\otimes H$ satisfying
	\begin{equation*}
		(\rho\otimes\mathrm{id}_H)\circ\rho=(\mathrm{id}_V\otimes\Delta)\circ\rho,\quad (\mathrm{id}_V\otimes\epsilon)\circ\rho=\mathrm{id}_V .
	\end{equation*}
	We use a Sweedler-like notation $\rho(v)=v_{(0)}\otimes v_{(1)}$ for a coaction.
	
	
	In addition a \textbf{right $H$-comodule algebra} is a right comodule $(A,\rho)$, where $A$ is an algebra and the coaction $\rho$ is an algebra morphism. Here the algebra structure on $A\otimes H$ is the one given by the component-wise product.  The elements for which the coaction is trivial, 
	\begin{equation*}
		V^{coH}:=\{v\in V|\rho(v)=v\otimes 1_H\} ,
	\end{equation*}
	form the space of \textbf{coaction invariant elements}. For a right $H$-comodule algebra $A$ we have that $A^{coH}$ is a subalgebra of $A$. 
	
	Let $(V,\rho_V)$ and $(W,\rho_W)$ be right $H$-comodules, a linear map $f:V\to W$ is said to be \textbf{right $H$-colinear} or a right comodule morphism if
	\begin{equation*}
		\rho_W\circ f=(f\otimes\mathrm{id}_H)\circ\rho_V .
	\end{equation*}
	A $H$-colinear map between $H$-comodule algebras which in addition is an algebra morphism is said to be a $H$-comodule algebra morphism.
	
	Right $H$-comodules with $H$-colinear maps form a category $\mathfrak{M}^H$. The latter is monoidal meaning that the tensor product $V\otimes W$ with $V, W\in \mathfrak{M}^H$ is a right $H$-comodule. The coaction is the \textbf{diagonal} one.
	\begin{equation}
		\rho^{\otimes}:V\otimes W\longrightarrow V\otimes W\otimes H,\quad v\otimes w\longmapsto v_{(0)}\otimes w_{(0)}\otimes v_{(1)}w_{(1)}.
	\end{equation}

	\begin{expl}
		Any Hopf algebra ($H, \Delta $) is a right $H$-comodule algebra with coaction $\rho=\Delta$. In this case one has $H^{coH}=\mathbb{K}$ : from $h_{(1)}\otimes h_{(2)}=h\otimes 1_H$ 
		by applying $\epsilon \otimes \mathrm{id}_H$
		on both sides of the latter we get $h=\epsilon(h) \in \mathbb{K}$.
	\end{expl}
	
	\subsection{Hopf--Galois extensions}
	Let $(A,\rho)$ be a right $H$-comodule algebra with  $B:=A^{coH}$. The \textit{canonical map} is defined as
	\begin{equation}
		\begin{split}
			\chi:=(\mathrm{m}_A\otimes\mathrm{id}_H)\circ(\mathrm{id}_A\otimes_B\rho) : A\otimes_B A  & \longrightarrow A\otimes H\\
			a\otimes_B\tilde{a}  & \longmapsto a\tilde{a}_{(0)}\otimes \tilde{a}_{(1)} ,
		\end{split} 
	\end{equation}
	whose domain is the \textbf{balanced tensor product} $A\otimes _BA$.
	The map $\chi$ is a left $A$-linear and right $H$-colinear morphism. The inclusion $B\subseteq A$ is said to be a \textbf{$H$-Hopf--Galois extension} when the associated canonical map $\chi$ is invertible.   
	
	\begin{remark}
		\label{rem:canonical_maps_equivalence}
		One can associate a different canonical map to $B\subseteq A$, namely
		\begin{equation*}
			\chi':A\otimes_BA\longrightarrow A\otimes H,\quad a\otimes_B\tilde{a}\longmapsto a_{(0)}\tilde{a}\otimes a_{(1)}.
		\end{equation*} 
		If we assume that the antipode of $H$ is bijective then we that $\varphi:A\otimes H\to A\otimes H$, mapping $a\otimes h\to a_{(0)}\otimes S^{-1}(h)a_{(1)}$, this is bijective too and $\chi=\varphi\circ\chi'$. Thus the invertibility of $\chi$ is equivalent to the invertibility of $\chi'$. 
	\end{remark}
	In the present paper we restrict our attention to the following class of Hopf--Galois extensions \cite{hajac1996strong,schneider1990principal}.  
	\begin{defn}
		\label{defn:principal_comodule_algebra}
		Let $H$ be a Hopf algebra with bijective antipode. A $H$-Hopf--Galois extension  is called a \textbf{principal $H$-comodule algebra} when $A$ is a left faithfully flat $B$-module. 
	\end{defn}    
	
	\begin{remark}
		We recall that a (left) module $M$ over an algebra $B$ is said to be faithfully flat if any sequence  $N_1\to N_2\to N_3$ of right $B$-modules is exact if and only if the sequence $N_1\otimes_R M\to N_2\otimes_R M\to N_3\otimes_R M$ is such (the latter is a sequence of vector space). Principal comodule algebras have been characterized in \cite{schneider1990principal} and for them the surjectivity of the canonical map is equivalent to its invertibility. 
	\end{remark}
	
	Using the bijectivity of the canonical map of a $H$-Hopf--Galois extension, one defines the \textbf{translation map}, 
	\begin{equation}
		\tau:=\chi^{-1}|_{1_A\otimes H}:H\longrightarrow A\otimes_B A . 
	\end{equation}
	In the following, we use a Sweedler-like convention for this map and write
	\begin{align*}
		\tau(h)=h^{\langle1\rangle}\otimes_Bh^{\langle2\rangle},\quad h\in H .
	\end{align*}
	Thus, by definition,
	\beq\label{tuno}
	\tuno{h} \zero{\tdue{h}} \otimes \one{\tdue{h}} = 1_A \otimes h.
	\eeq
	
	The translation map enjoys a series of useful properties that we list below.
	\begin{prop}
		\label{translation_map_properties}
		\cite{Schneider1990RepresentationTO,Brzezi_ski_1996} Let $H$ be a Hopf algebra with counit $\epsilon$ and antipode $S$ and $A$ a right $H$-comodule algebra such that $B\subseteq A$ is a $H$-Hopf--Galois extension. Then the translation map $\tau$ satisfies, for any $h,k\in H$ and $a\in A$, the properties:
		\begin{align}
			\label{eq:transl_1}
			h^{\langle1\rangle}\otimes_B\tensor{h}{^{\langle2\rangle}_{(0)}}\otimes\tensor{h}{^{\langle2\rangle}_{(1)}}&=\tensor{h}{_{(1)}^{\langle1\rangle}}\otimes_B\tensor{h}{_{(1)}^{\langle2\rangle}}\otimes h_{(2)},\\
			\label{eq:transl_2}
			\tensor{h}{^{\langle1\rangle}_{(0)}}\otimes_Bh^{\langle2\rangle}\otimes\tensor{h}{^{\langle1\rangle}_{(1)}}&=\tensor{h}{_{(2)}^{\langle1\rangle}}\otimes_B\tensor{h}{_{(2)}^{\langle2\rangle}}\otimes S(h_{(1)}),\\
			\label{eq:transl_3}
			h^{\langle1\rangle}h^{\langle2\rangle}&=\epsilon(h)1_A,\\
			\label{eq:transl_4}
			(hk)^{\langle1\rangle}\otimes_B(hk)^{\langle2\rangle}&=k^{\langle1\rangle}h^{\langle1\rangle}\otimes_Bh^{\langle2\rangle}k^{\langle2\rangle},\\
			\label{eq:transl_5}
			a_{(0)}\tensor{a}{_{(1)}^{\langle1\rangle}}\otimes_B\tensor{a}{_{(1)}^{\langle2\rangle}}&=1_A\otimes_Ba.
		\end{align}
	\end{prop}

	\begin{expl}
		\label{expl:antipode_canonical}
		A bialgebra is a Hopf algebra if and only if the extension $\mathbb{K}\subseteq H$ is $H$-Hopf--Galois. Indeed, if there is an antipode $S$ the canonical map is invertible via
		\begin{equation*}
			\chi^{-1}(h\otimes k)=hS(k_{(1)})\otimes k_{(2)} .
		\end{equation*}
		Conversely, if $H$ is a bialgebra and $\mathbb{K}\subseteq H$ is Hopf--Galois then the map
		\begin{equation*}
			S:H\longrightarrow H,\quad h\longmapsto h^{\langle1\rangle}\epsilon(h^{\langle2\rangle}),
		\end{equation*}
		is the antipode that makes $H$ a Hopf algebra.
	\end{expl}
	
	For completeness, we state and prove a lemma which we need later on.
	\begin{lem}
		\label{lem:equiv_of _coinvarinats}
		Let $B\subseteq A$ be a right $H$-Hopf--Galois extension and define
		\begin{align*}
			\mathcal{L}_A&:=\{a\otimes \tilde{a}\in A\otimes A|a\otimes \tilde{a}_{(0)}\otimes \tilde{a}_{(1)}=a_{(0)}\otimes \tilde{a}\otimes S(a_{(1)})\},\\
			(A\otimes A)^{coH}&:=\{a\otimes \tilde{a}\in A\otimes A|a_{(0)}\otimes \tilde{a}_{(0)}\otimes a_{(1)}\tilde{a}_{(1)}=a\otimes \tilde{a}\otimes 1_H\},\\
			\mathcal{C}(A,H)&:=\{a\otimes \tilde{a}\in A\otimes A|a_{(0)}\otimes\tensor{a}{_{(1)}^{\langle1\rangle}}\otimes_B\tensor{a}{_{(1)}^{\langle2\rangle}}\tilde{a}=a\otimes \tilde{a}\otimes_B1_A\} .
		\end{align*} 
		Then one has $\mathcal{L}_A=(A\otimes A)^{coH}=\mathcal{C}(A,H)$.
	\end{lem}
	\begin{proof}
		We first prove the inclusion $\mathcal{L}_A\subseteq(A\otimes A)^{coH}$. By taking $a\otimes \tilde{a}\in\mathcal{L}_A$ and applying on both sides of the defining equation the map $\rho\otimes\mathrm{id}_H$ we get
		\begin{equation*}
			a_{(0)}\otimes a_{(1)}\otimes \tilde{a}_{(0)}\otimes \tilde{a}_{(1)}=a_{(0)}\otimes a_{(1)}\otimes \tilde{a}\otimes S(a_{(2)}) .
		\end{equation*}
		By applying again on both sides the map $(\mathrm{id}_{A\otimes A}\otimes\mathrm{m}_H)\circ(\mathrm{id}_{A}\otimes\mathrm{flip}_{HA}\otimes\mathrm{id}_H)$:
		\begin{equation*}
			a_{(0)}\otimes \tilde{a}_{(0)}\otimes a_{(1)}\tilde{a}_{(1)}=a_{(0)}\otimes \tilde{a}\otimes a_{(1)}S(a_{(2)})=a\otimes \tilde{a}\otimes 1_H,
		\end{equation*}
		where we used $a_{(1)}S(a_{(2)})=\epsilon(a_{(1)})1_H$ and $a_{(0)}\epsilon(a_{(1)})=a$.
		This proves the inclusion. We next prove that $(A\otimes A)^{coH}\subseteq\mathcal{L}_A$. With $a\otimes \tilde{a}\in (A\otimes A)^{coH}$ by applying $(\mathrm{id}_A\otimes S\circ\rho)\otimes\mathrm{id}_{A\otimes H}$ on both sides of the defying equation we get
		\begin{equation*}
			a_{(0)}\otimes S(a_{(1)})\otimes \tilde{a}_{(0)}\otimes a_{(2)}\tilde{a}_{(1)}=a_{(0)}\otimes S(a_{(1)})\otimes \tilde{a}\otimes 1_H .
		\end{equation*}
		By applying again $(\mathrm{id}_{A\otimes A}\otimes\mathrm{m}_H)\circ(\mathrm{id}_{A}\otimes\mathrm{flip}_{HA}\otimes\mathrm{id}_H)$ we have the equation
		\begin{equation*}
			a_{(0)}\otimes \tilde{a}_{(0)}\otimes S(a_{(1)})a_{(2)}\tilde{a}_{(1)}=a_{(0)}\otimes \tilde{a}\otimes S(a_{(1)}).
		\end{equation*}
		The left hand side of the latter reduces to $a\otimes \tilde{a}_{(0)}\otimes \tilde{a}_{(1)}$ once we use $a_{(1)}S(a_{(2)})=\epsilon(a_{(1)})1_H$ and $a_{(0)}\epsilon(a_{(1)})=a$. We conclude that $(A\otimes A)^{coH}=\mathcal{L}_A$.
		
		To prove the other equivalence we use bijectivity of the canonical map $\chi$. Take $a\otimes \tilde{a}\in \mathcal{C}(A,H)$ and apply $\mathrm{id}_A\otimes\chi$ on both sides of the defying equation to get
		\begin{equation*}
			a_{(0)}\otimes\tensor{a}{_{(1)}^{\langle1\rangle}}\tensor{a}{_{(1)}^{\langle2\rangle}_{(0)}}\tilde{a}_{(0)}\otimes\tensor{a}{_{(1)}^{\langle2\rangle}_{(1)}}\tilde{a}_{(1)}=a\otimes \tilde{a}\otimes1_H .
		\end{equation*}
		Using \eqref{tuno},	
		the expression on the left hand side reduces to $a_{(0)}\otimes \tilde{a}_{(0)}\otimes a_{(1)}\tilde{a}_{(1)}$; thus $\mathcal{C}(A,H)\subseteq(A\otimes A)^{coH}$. Conversely, take $a\otimes \tilde{a}\in(A\otimes A)^{coH}$ and apply $\mathrm{id}_A\otimes\chi^{-1}$ to both sides to the coinvariance equation. The left $A$-linearity of $\chi^{-1}$ yields
		\begin{equation*}
			a_{(0)}\otimes \tilde{a}_{(0)}(a_{(1)}\tilde{a}_{(1)})^{\langle1\rangle}\otimes_B(a_{(1)}\tilde{a}_{(1)})^{\langle2\rangle}=a\otimes \tilde{a}\otimes_B1_A  .
		\end{equation*}
		Using eqs. \eqref{eq:transl_4} and \eqref{eq:transl_5}, 
		the left hand side can be rewritten as $a_{(0)}\otimes\tau(a_{(1)})\tilde{a}$, showing the converse inclusion. So we have $(A\otimes A)^{coH}=\mathcal{C}(A,H)$.
	\end{proof}
	
	\begin{remark}
		Since in the first equivalence we do not use that $A$ is an algebra nor that $B\subseteq A$ is a Hopf--Galois extension, the equality $\mathcal{L}_V=(V\otimes V)^{coH}$ holds for any right $H$-comodule $V$.
	\end{remark}

	\section{Weak and full Hopf algebroids}
	In this section, we review the general theory of bialgebroids and Hopf algebroids in the sense of both Schauenburg \cite{schau_dd} and B\"ohm--Szlachanyi \cite{bohm2004hopf}. Then we focus our attention on the Ehresmann--Schauenburg bialgebroid associated to a Hopf--Galois extension and finally we study the concept of a twist of an antipode.
	
	\subsection{Bialgebroids, canonical maps and antipodes}
	As their name suggests, bialgebroids are a generalization of bialgebras. The latter are algebras over a field that are equipped with a compatible coalgebra structure. Roughly speaking, a bialgebroid is a bialgebra over a (possibly noncommutative) algebra rather than over a field. We introduce bialgebroids in the same spirit of bialgebras, that is by giving the definition of an algebra, a coalgebra and then the compatibility conditions.
	
	In this setting, the algebra role is played by:
	\begin{defn}
		\label{defn:ring}
		Let $B$ be an algebra and denote by $B^{e}:=B\otimes B^{op}$ its enveloping algebra. A \textbf{$B^{e}$-ring} is the datum of $(U,t,s)$ where
		\begin{itemize}
			\item $U$ is an algebra;
			\item $s:B\to U$ is an algebra morphism, $t:B^{op}\to U$ is an anti-algebra morphism;
			\item the ranges of $s$ and of $t$ commute in $U$: $s(b)t(b)=t(b)s(b)$ for all $b\in B$.
		\end{itemize}
		We refer to $s$ and $t$ as \textit{source} and \textit{target} map respectively.
	\end{defn}  
	
	A $B$-bimodule structure is then defined on the algebra $U$ by 
	\begin{equation}
		\label{eq:bimod_strc_ring}
		bub':=s(b)t(b')u,\quad u\in U,\quad b,b'\in B .
	\end{equation}
	A balanced tensor product over the algebra $B$ is then obtained as the quotient of $U\otimes U$ by the ideal generated by elements of the form $t(b)u\otimes u'-u\otimes s(b)u'$ with $u,u'\in U$ and $b\in B$: 
	\begin{equation}
		\label{eq:tensor_B}
		U\otimes_BU:=U\otimes U/(t(b)u\otimes u'-u\otimes s(b)u') .
	\end{equation}
	With the source and target maps we can endow $U$ also with a $B^{op}$-bimodule structure.
	If we denote by $^{op}:B\to B^{op}$ the linear anti-multiplicative isomorphism, then we have for any $u\in U$ and $b,b'\in B$ 
	\begin{equation}
		\label{eq:B_opp_bimod1}
		b^{op}\cdot u\cdot b'^{op}:=t(b)ut(b'),
	\end{equation} 
	\begin{equation}
		\label{eq:B_opp_bimod2}
		b^{op}*u*b'^{op}:=s(b')us(b).
	\end{equation}
	These are used to get different balanced tensor products. That is we define
	\begin{align}
		\label{eq:tensor_prod_B_opp1}
		U\odot_{B^{op}}U:=U\otimes U/(ut(b)\otimes u'-u\otimes t(b)u'),\\
		\label{eq:tensor_prod_B_opp2}
		U\circledast_{B^{op}}U:=U\otimes U/(s(b)u\otimes u'-u\otimes u's(b)) .
	\end{align}
	
	The notion of coalgebra is extended by the following. 
	\begin{defn}
		\label{defn:coring}
		Let $B$ be an algebra, a \textbf{$B$-coring} is the datum of $(\mathcal{C},\underline{\Delta},\underline{\epsilon})$ where
		\begin{itemize}
			\item $\mathcal{C}$ is a $B$-bimodule;
			\item $\underline{\Delta}:\mathcal{C}\longrightarrow \mathcal{C}\otimes_B \mathcal{C}$ is a $B$-bimodule morphism satisfying
			\begin{equation}
				(\underline{\Delta}\otimes_B\mathrm{id}_{\mathcal{C}})\circ\underline{\Delta}=(\mathrm{id}_{\mathcal{C}}\otimes_B\underline{\Delta})\circ\underline{\Delta};
			\end{equation}
			\item $\underline{\epsilon}$ is a $B$-bimodule morphism with the property
			\begin{equation}
				\label{eq:counital_coring}
				(\underline{\epsilon}\otimes_B\mathrm{id}_{\mathcal{C}})\circ\underline{\Delta}=\mathrm{id}_{\mathcal{C}}=(\mathrm{id}_{\mathcal{C}}\otimes_B\underline{\epsilon})\circ\underline{\Delta} \, .
			\end{equation}
		\end{itemize}
		We call $\underline{\Delta}$ the \textit{coproduct} and $\underline{\epsilon}$ the \textit{counit} as in the case of bialgebras.
	\end{defn}
	
	Then these two structures are put together. 
	\begin{defn}
		\label{defn:bialgebroid}
		Let $B$ be an algebra, the datum of 
		$(\mathcal{H},s,t, \underline{\Delta},\underline{\epsilon})$ is a (left) \textbf{$B$-bialgebroid} when:
		\begin{itemize}
			\item The triple $(\mathcal{H},s,t)$ is a $B^{e}$-ring;
			\item[~]
			\item The triple $(\mathcal{H},\underline{\Delta},\underline{\epsilon})$ is a $B$-coring 
			for the $B$-bimodule structure of eq.~\eqref{eq:bimod_strc_ring};
			\item[~]
			\item The coproduct $\underline{\Delta}$ is an algebra morphism when corestricted to
			\begin{equation}
				\label{eq:Takeuchi_prod}
				\mathcal{H}\times_B\mathcal{H}:=\{h\otimes_Bh'\in\mathcal{H}\otimes_B\mathcal{H} \, | \,  ht(b)\otimes_Bh'=h\otimes_Bh's(b),\forall b\in B\} ; 
			\end{equation}
			\item The counit is unital $\underline{\epsilon}(1_{\mathcal{H}})=1_A$ and satisfies
			\begin{equation}
				\label{eq:counit_product}
				\underline{\epsilon}(hh')=\underline{\epsilon}(hs(\underline{\epsilon}(h')))=\underline{\epsilon}(ht(\underline{\epsilon}(h'))),\quad \forall h,h'\in\mathcal{H} \, .
			\end{equation}
		\end{itemize}
	\end{defn}
	\noindent
	In the above eq.~\eqref{eq:Takeuchi_prod}, we have the Takeuchi product \cite{tak_gpalg} that is an algebra with component-wise multiplication and unit $1_{\mathcal{H}}\otimes 1_{\mathcal{H}}$. This condition is important since in general 
	$\mathcal{H}\otimes_B\mathcal{H}$ has no algebra structure, 
	unless $\mathcal{H}$ is a symmetric $B$-bimodule. 
	
	We use the Sweedler summation notation for the coproduct $\underline{\Delta}(h)=h_{(1)}\otimes_B h_{(2)}$ and by recalling the $B$-bimodule structure \eqref{eq:bimod_strc_ring} one has for 
	eq.~\eqref{eq:counital_coring} 
	\begin{equation}
		\label{eq:counital_bialgebroid}
		\underline{\epsilon}(h_{(1)}) h_{(2)}=s(\underline{\epsilon}(h_{(1)}))h_{(2)}=h=t(\underline{\epsilon}(h_{(2)}))h_{(1)}=h_{(1)} \underline{\epsilon}(h_{(2)}) .
	\end{equation}
	
	For a left bialgebroid $\mathcal{H}$ over an algebra $B$, the category of left $\mathcal{H}$-module $_{\mathcal{H}}\mathfrak{M}$ is a monoidal category with respect to the tensor product $\otimes_B$. In contrast, the category of right $\mathcal{H}$-module $\mathfrak{M}_{\mathcal{H}}$ is not monoidal in general. Using a left-right symmetry argument one can define a right bialgebroid over $B$ \cite{brz-wis}. Given a left (right) bialgebroid it is possible that a right (left) bialgebroid structure cannot be defined. We return to this point later in the text.  
	
	We saw in example \ref{expl:antipode_canonical} that the existence of the antipode for a bialgebra $H$ is equivalent to the Hopf--Galois condition for $\mathbb{K}\subseteq H$.  This motivates the following.
	\begin{defn}
		\label{defn:Hopf_algebroid}
		A $B$-bialgebroid $\mathcal{H}$ is a (left) \textbf{weak Hopf algebroid over $B$} if the (canonical) map
		\begin{equation}
			\label{eq:canonical_bialgeb}
			\beta:\mathcal{H}\odot_{B^{op}}\mathcal{H}\longrightarrow\mathcal{H}\otimes_B\mathcal{H},\quad h\odot_{B^{op}}h'\longmapsto h_{(1)}\otimes_B h_{(2)}h' \, ,
		\end{equation}
		is bijective. Notice that the map is well-defined over the tensor product \eqref{eq:tensor_prod_B_opp1} since $\underline{\Delta}(t(b))=1\otimes_B t(b)$ for every $b\in B$.  
	\end{defn} 
	
	This condition generalizes the (left version of)  example \ref{expl:antipode_canonical}, that is if $B=\mathbb{K}$ we retrieve the Hopf algebra case. This definition was originally introduced in \cite{schau_dd} where the invertibility of $\beta$ is proved to be equivalent to the inner-hom functor preserving property of the forgetful functor 
	$_{\mathcal{H}}\mathfrak{M} \to{}_{B}\mathfrak{M}_{B}$ from the category of left $\mathcal{H}$-modules to the one of $B$-bimodules.  
	
	ln parallel with Remark \ref{rem:canonical_maps_equivalence}, 
	there is another canonical map associated to a bialgebroid, that is, 
	\begin{equation}
		\label{eq:canonical_bialgeb2}
		\lambda:\mathcal{H}\circledast_{B^{op}}\mathcal{H}\longrightarrow\mathcal{H}\otimes_B\mathcal{H},\quad h\circledast_{B^{op}}h'\longmapsto h'_{(1)}h\otimes_Bh'_{(2)},
	\end{equation}
	using the tensor product \eqref{eq:tensor_prod_B_opp2}, which is well-defined due to $\underline{\Delta}(s(b))=s(b)\otimes_B1_{\mathcal{H}}$.
	
	An alternative way to define a Hopf algebroid was proposed in \cite{bohm2019alternative,bohm2004hopf}, closer in spirit to the Hopf algebra case.
	\begin{defn}
		Let $(\mathcal{H}, s,t, \underline{\Delta},\underline{\epsilon})$ be a bialgebroid over an algebra $B$. An \textbf{antipode} for it is the datum of a bijective anti-algebra morphism $\underline{S}:\mathcal{H}\to\mathcal{H}$, with inverse $\underline{S}^{-1}:\mathcal{H}\to\mathcal{H}$, such that,
		\begin{align}
			\label{eq:antipode_1}
			\underline{S}\circ t=s
		\end{align} and, for all $h\in\mathcal{H}$,
		\begin{align}
			\label{eq:antipode_2}
			\underline{S}(h_{(1)})_{(1')}h_{(2)}\otimes_B \underline{S}&(h_{(1)})_{(2')}=1_{\mathcal{H}}\otimes_B \underline{S}(h),\\
			\label{eq:antipode_3}
			\underline{S}^{-1}(h_{(2)})_{(1')}\otimes_B\underline{S}^{-1}&(h_{(2)})_{(2')}h_{(1)}=\underline{S}^{-1}(h)\otimes_B 1_{\mathcal{H}}. 
		\end{align}
		
		\noindent
		A $B$-bialgebroid equipped with an antipode is called a \textbf{full Hopf algebroid} over $B$ and here it is denoted by $(\mathcal{H},\underline{S})$.
	\end{defn} 
	
	\begin{remark}
		One notices that eq.\eqref{eq:antipode_1} implies that for the inverse it holds that
		\begin{equation}
			\underline{S}^{-1}\circ s=t .
		\end{equation}
		Moreover eqs.\eqref{eq:antipode_2} and \eqref{eq:antipode_3} imply that the following equations hold for all $h\in\mathcal{H}$
		\begin{align}
			\label{eq:antipode_5}
			\underline{S}(h_{(1)})h_{(2)}&=(t\circ\underline{\epsilon}\circ\underline{S})(h),\\
			\label{eq:antipode_6}
			\underline{S}^{-1}(h_{(2)})h_{(1)}&=(s\circ\underline{\epsilon}\circ\underline{S}^{-1})(h).
		\end{align}
	\end{remark}
	
	We call $(\mathcal{H},\underline{S})$ a full algebroid to distinguish it from Definition~\ref{defn:Hopf_algebroid} and the 
	names are justified by the following result 
	\cite{han2023hopf}.
	\begin{prop}
		\label{prop:full_is_Hopf}
		Let $(\mathcal{H},\underline{S})$ be a full Hopf algebroid over $B$. Then $\mathcal{H}$ is a Hopf algebroid in the sense of Definition~\ref{defn:Hopf_algebroid}.
	\end{prop}
	\begin{proof}
		We have to show that the canonical map $\beta$ in Definition~\ref{defn:Hopf_algebroid} is linearly bijective. Let us define \begin{equation*}
			\widetilde{\beta}:\mathcal{H}\otimes_B\mathcal{H}\longrightarrow\mathcal{H}\odot_{B^{op}}\mathcal{H},\quad h\otimes_B h'\longmapsto \underline{S}^{-1}(\underline{S}(h)_{(2)})\odot_{B^{op}}\underline{S}(h)_{(1)}h' .
		\end{equation*}
		We claim that $\widetilde{\beta}=\beta^{-1}$.
		We first check that the expression of $\widetilde{\beta}$ agrees with the coproduct. If $b\in B$ and $h,h'\in\mathcal{H}$, then
		\begin{align*}
			\underline{S}^{-1}(s(b)\underline{S}(h)_{(2)})\odot_{B^{op}}\underline{S}(h)_{(1)}h'&=\underline{S}^{-1}(\underline{S}(h)_{(2)})t(b)\odot_{B^{op}}\underline{S}(h)_{(1)}h'\\
			&=\underline{S}^{-1}(\underline{S}(h)_{(2)})\odot_{B^{op}}t(b)\underline{S}(h)_{(1)}h'.
		\end{align*}
		This follows from the anti-algebra property of $\underline{S}$ and the definition of $\odot_{B^{op}}$ of \eqref{eq:tensor_prod_B_opp1}.
		Moreover, we have to check that it is well-defined on $\otimes_B$ $\mathcal{H}\odot_{B^{op}}\mathcal{H}$
		\begin{align*}
			\widetilde{\beta}(t(b)h\otimes_Bh')&=\underline{S}^{-1}(\underline{S}(t(b)h)_{(2)})\odot_{B^{op}}\underline{S}(t(b)h)_{(1)}h'\\
			&=\underline{S}^{-1}((\underline{S}(h)s(b))_{(2)})\odot_{B^{op}}(\underline{S}(h)s(b))_{(1)}h'\\
			&=\underline{S}^{-1}(\underline{S}(h)_{(2)})\odot_{B^{op}}\underline{S}(h)_{(1)}s(b)h'=\widetilde{\beta}(h\otimes_Bs(b)h')
		\end{align*}
		for all $b\in B$ and $h,h'\in\mathcal{H}$, where we used the antialgebra property of $\underline{S}$, the property \eqref{eq:antipode_1} and that 
		$\underline{\Delta}(s(b))=s(b)\otimes_B1_{\mathcal{H}}$.
		Then,
		\begin{align*}
			(\widetilde{\beta}\circ\beta)(h\odot_{B^{op}}h')&=\underline{S}^{-1}(\underline{S}(h_{(1)})_{(2')})\otimes_B\underline{S}(h_{(1)})_{(1')}h_{(2)}h'\\
			&=\underline{S}^{-1}(\underline{S}(h))\otimes_Bh'=h\otimes_Bh' , 
		\end{align*}
		\begin{align*}
			(\beta\circ\widetilde{\beta})(h\otimes_Bh')&=\underline{S}^{-1}(\underline{S}(h)_{(2)})_{(1')}\odot_{B^{op}}\underline{S}^{-1}(\underline{S}(h)_{(2)})_{(2')}\underline{S}(h)_{(1)}h'\\
			&=\underline{S}^{-1}(\underline{S}(h))\odot_{B^{op}}h'=h\odot_{B^{op}}h'
		\end{align*}
		for all $h,h'\in\mathcal{H}$, where we have used \eqref{eq:antipode_2} and \eqref{eq:antipode_3} respectively.
	\end{proof}
	
	\begin{remark}
		One can show that if an invertible antipode exists then the map \eqref{eq:canonical_bialgeb2} is bijiective too, with its inverse given by
		\begin{equation}
			\label{eq:invertible_can2}
			\lambda^{-1}(h\otimes_Bh')=\underline{S}^{-1}(h')_{(2)}h\circledast_{B^{op}}\underline{S}(\underline{S}^{-1}(h')_{(1)}) .
		\end{equation}
		Furthermore the simultaneous bijectivity of $\beta$ and $\lambda$ is equivalent to the existence of an invertible antipode for a left bialgebroid with a compatible right bialgebroid structure \cite{bohm2004hopf}. As mentioned before this is not always the case. 
		Indeed there are situations where an antipode does not exist at all. We refer to \cite{rovi2014hopf} for one such example. 
	\end{remark}
	
	Having an invertible antipode one can give a $B^{op}$-bimodule structure to $\mathcal{H}$ that differs from the ones in \eqref{eq:B_opp_bimod1} and \eqref{eq:B_opp_bimod2}. This is given by
	\begin{equation}
		b^{op}
		hb'^{op}=h\underline{S}^{-1}(t(b))t(b') 
	\end{equation}
	with associated balanced tensor product over $B^{op}$:
	\begin{equation}
		\mathcal{H}\otimes_{B^{op}}\mathcal{H}:=\mathcal{H}\otimes\mathcal{H}/(ht(b)\otimes h'-h\otimes \underline{S}^{-1}(t(b))h').
	\end{equation}
	
	\begin{lem}\label{lem:right_coprod}
		Let $(\mathcal{H},\underline{S})$ be a full Hopf algebroid over $B$, and consider the map from $\mathcal{H}$ to $\mathcal{H}\otimes_{B^{op}}\mathcal{H}$ given by the formula
		\begin{equation}
			\label{eq:right_coprod}
			h \longmapsto h^{[1]}\otimes_{B^{op}}h^{[2]}:=\underline{S}(\underline{S}^{-1}(h)_{(2)})\otimes_{B^{op}}\underline{S}(\underline{S}^{-1}(h)_{(1)}) .
		\end{equation}
		This has the following properties:
		\begin{align}
			\label{eq:right_coprod_prop1}
			h^{[1]}\underline{S}(h^{[2]})&=s(\underline{\epsilon}(h)),\\
			\label{eq:right_coprod_prop2}
			\lambda^{-1}(h\otimes_Bh')&=\underline{S}^{-1}(h'^{[1]})h\circledast_{B^{op}}h'^{[2]},\\
			\label{eq:right_coprod_prop3}
			h^{[1]}\otimes_{B^{op}}\tensor{h}{^{[2]}_{(1)}}\otimes_B\tensor{h}{^{[2]}_{(2)}}&=\tensor{h}{_{(1)}^{[1]}}\otimes_{B^{op}}\tensor{h}{_{(1)}^{[2]}}\otimes_Bh_{(2)},\\
			\label{eq:right_coprod_prop4}
			\underline{S}(h)^{[1]}\otimes_{B^{op}}\underline{S}(h)^{[2]}&=\underline{S}(h_{(2)})\otimes_{B^{op}}\underline{S}(h_{(1)}).
		\end{align}
		for all $h,h'\in\mathcal{H}$. Here in \eqref{eq:right_coprod_prop2} we have the inverse of the canonical map \eqref{eq:canonical_bialgeb2}.
	\end{lem}
	\begin{proof}
		For the first equation, we use the anti-multiplicative property of $\underline{S}$ and \eqref{eq:antipode_5}. For the second one just rewrite eq.~\eqref{eq:invertible_can2} using the definition \eqref{eq:right_coprod}. The third equation is proved by applying on both sides the isomorphism $\lambda\otimes_B\mathrm{id}_{\mathcal{H}}$ and showing that one gets the same result. The last equation comes from evaluating \eqref{eq:right_coprod} at $k=\underline{S}(h)$ with $h\in \mathcal{H}$. 
	\end{proof}
	\subsection{Twists}
	In this subsection we introduce the notion of a twist for a bialgebroid. Since we are working on a bimodule over an algebra, we must specify which module structure we refer to. We exclusively work with the right $B$-module structure of \eqref{eq:bimod_strc_ring} given by the action of the target map. 
	
	The convolution algebra $\mathcal{H}^*$ of a $B$-bialgebroid $\mathcal{H}$ is the algebra of right $B$-linear morphisms $\phi_*: \mathcal{H} \to B$ with product given by
	\begin{equation}
		(\phi_*\psi_*)(h):=\psi_*\left(s(\phi_*(h_{(1)}))h_{(2)}\right),\quad h\in\mathcal{H}.
	\end{equation}
	The unit of this algebra is the counit $\underline{\epsilon}$ of $\mathcal{H}$. 
	(For more details see \cite{sweedler1975predual}). 
	There is a natural right $\mathcal{H}^*$-action on $\mathcal{H}$ given by
	\begin{equation}
		\label{eq:right_H_*_action}
		h\triangleleft\phi_*:=\left(s(\phi_*(h_{(1)}))\right)h_{(2)},\quad h\in\mathcal{H}, \quad\phi_*\in\mathcal{H}^*.
	\end{equation}
	In what follows we focus on a particular subgroup contained in the unit group of $\mathcal{H}^*$, the one of twists. 
	\begin{defn}
		\label{defn:twists}
		Let $\mathcal{H}$ be a $B$-bialgebroid, we say that an element $\phi_*\in\mathcal{H}^*$ is a \textbf{twist} if it is in the unit group $(\mathcal{H}^*)^+$ of $\mathcal{H}^*$ and moreover it is such that
		\begin{align}
			\label{eq:twist_unital}
			1_\mathcal{H}\triangleleft\phi_* & =1_\mathcal{H}, \\
			\label{eq:twist_multiplicative}
			(h\triangleleft\phi_*)(h'\triangleleft\phi_*) & = hh'\triangleleft\phi_*,\quad\forall h,h\in\mathcal{H} .
		\end{align}
		The counit of $\mathcal{H}$ is the neutral element of this group. For the rest of the paper, we denote the group of twists by $\mathcal{T}^*$. 
	\end{defn}
	

	\begin{remark}\label{renos}
		In the original definition of twists in \cite{bohm2019alternative} there was an additional requirement:
		\begin{equation*}
			\underline{S}(h_{(1)})\triangleleft\phi_*\otimes'_{B}h_{(2)}=\underline{S}(h_{(1)})\otimes'_{B}h_{(2)}\triangleleft\phi^{-1}_* . 
		\end{equation*}
		This involves the balanced tensor product $\otimes'_{B}$ obtained from $\mathcal{H}\otimes\mathcal{H}$ over the ideal generated by $hs(b)\otimes h'-h\otimes s\circ\phi_*^{-1}\circ s(b)h'$, for $b\in B$ and $h, h'\in\mathcal{H}$. For the present paper a twist does not need to fulfill this property and, in particular, we have that the notion of twist makes sense even for bialgebroids.
		
	\end{remark}
	
	\begin{thm}
		\label{thm:Bohm_thm}
		\cite{bohm2019alternative} Let $(\mathcal{H},\underline{S})$ be a full Hopf algebroid over an algebra $B$, then $(\mathcal{H},\underline{S}')$ is a Hopf algebroid with the same underlying $B$-bialgebroid structure if and only if there exists a twist $\phi_*\in\mathcal{T}^*$ such that
		\begin{equation*}
			\underline{S}'(h)=\underline{S}(h\triangleleft\phi_*)\,\quad h\in\mathcal{H} .
		\end{equation*}
	\end{thm}
	\begin{proof}
		Our proof differs slightly from the one in \cite{bohm2019alternative}, thus we report it.
		
		Let $\phi_*\in\mathcal{T}^*$ and $\underline{S}$ be an antipode on $\mathcal{H}$. Then the map $\underline{S}'$ defined above is invertible with inverse
		\begin{equation*}
			\underline{S}'^{-1}(h):=\underline{S}(h)\triangleleft\phi^{-1} .
		\end{equation*}
		One checks easily that $\underline{S}'$ is an anti-algebra morphism since $\underline{S}$ is such and given property \eqref{eq:twist_multiplicative}. Using $t(b)\triangleleft\phi_*=b$ for every $b\in B$, one checks that $\underline{S}'\circ t=s$. Now to prove eqs. \eqref{eq:antipode_2} and \eqref{eq:antipode_3} we first notice that for any $\psi_*\in\mathcal{H}^*$ one has
		\begin{equation*}
			(h\triangleleft\psi_*)_{(1)}\otimes_B(h\triangleleft\psi_*)_{(2)}=\underline{\Delta}(h\triangleleft\psi_*)=h_{(1)}\triangleleft\psi_*\otimes_B h_{(2)},
		\end{equation*}
		being the coproduct a left $B$-module morphism. Thus, for any $h\in\mathcal{H}$ we compute,
		\begin{equation*}
			\begin{split}
				\underline{S}'(h_{(1)})_{(1)'}h_{(2)}\otimes_B\underline{S}'(h_{(1)})_{(2)'}&=\underline{S}(h_{(1)}\triangleleft\phi_*)_{(1)'}h_{(2)}\otimes_B\underline{S}(h_{(1)}\triangleleft\phi_*)_{(2)'}\\
				&=\underline{S}((h\triangleleft\phi_*)_{(1)})_{(1)'}(h\triangleleft\phi_*)_{(2)}\otimes_B\underline{S}((h\triangleleft\phi_*)_{(1)})_{(2)'}\\
				&=1_{\mathcal{H}}\otimes_B\underline{S}(h\triangleleft\phi*)=1_{\mathcal{H}}\otimes_B\underline{S}'(h);
			\end{split}
		\end{equation*}
		\begin{equation*}
			\begin{split}
				\underline{S}'^{-1}(h_{(2)})_{(1)'}\otimes_B\underline{S}'^{-1}(h_{(2)})_{(2)'}h_{(1)}&=(\underline{S}^{-1}(h_{(2)})\triangleleft\phi_*)_{(1)'}\otimes(\underline{S}^{-1}(h_{(2)})\triangleleft\phi_*)_{(2)'}h_{(1)}\\
				&=\underline{S}^{-1}(h_{(2)})_{(1)'}\triangleleft\phi_*\otimes_B\underline{S}^{-1}(h_{(2)})_{(2)'}h_{(1)}\\
				&=\underline{S}^{-1}(h)\triangleleft\phi_*\otimes_B1_{\mathcal{H}}=\underline{S}'^{-1}(h)\otimes_B1_{\mathcal{H}} .
			\end{split}
		\end{equation*}
		We conclude that $\underline{S}'$ is an invertible antipode on $\mathcal{H}$.
		
		On the other hand, consider $\underline{S}$ and $\underline{S}'$ to be two antipodes on $\mathcal{H}$. The map $\underline{\epsilon}\circ\underline{S}'^{-1}\circ\underline{S}$ from $\mathcal{H}\to B$ lies in $\mathcal{T}^*$. Right $B$-linearity is easy to check. The map $\underline{\epsilon}\circ\underline{S}^{-1}\circ\underline{S}'$ is the inverse one: indeed we have
		\begin{equation*}
			\begin{split}
				(\underline{\epsilon}\circ\underline{S}'^{-1}\circ\underline{S})(\underline{\epsilon}\circ\underline{S}^{-1}\circ\underline{S}')(h)&=\underline{\epsilon}\circ\underline{S}^{-1}\circ\underline{S}'\left[s(\underline{\epsilon}\circ\underline{S}'^{-1}\circ\underline{S}(h_{(1)}))h_{(2)}\right]\\
				&=\underline{\epsilon}\circ\underline{S}^{-1}\circ\underline{S}'\left[s\circ\underline{\epsilon}\circ\underline{S}'^{-1}(\underline{S}(h_{(1)}))h_{(2)}\right]\\
				&=\underline{\epsilon}\circ\underline{S}^{-1}\circ\underline{S}'\left[s\circ\underline{\epsilon}\circ\underline{S}'^{-1}(\underline{S}(h)^{[2]})\underline{S}^{-1}(\underline{S}(h)^{[1]})\right]\\
				&=\underline{\epsilon}\circ\underline{S}^{-1}\circ\underline{S}'\left[s\circ\underline{\epsilon}\circ\underline{S}'^{-1}(\underline{S}(h)^{[2]'})\underline{S}'^{-1}(\underline{S}(h)^{[1]'})\right]\\
				&=\underline{\epsilon}\circ\underline{S}^{-1}\left[\underline{S}(h)^{[1]'}\underline{S}'\circ s\circ\underline{\epsilon}\circ\underline{S}'^{-1}(\underline{S}(h)^{[2]'})\right]\\
				&=\underline{\epsilon}\circ\underline{S}^{-1}\left[\underline{S}(h)^{[1]'}\underline{S}'(\underline{S}\tensor{(h)}{^{[2]'}}_{(1)})\underline{S}\tensor{(h)}{^{[2]'}}_{(2)}\right]\\
				&=\underline{\epsilon}\circ\underline{S}^{-1}\left[\underline{S}\tensor{(h)}{_{(1)}^{[1]'}}\underline{S}'(\underline{S}\tensor{(h)}{_{(1)}}^{[2]'})\underline{S}(h)_{(2)}\right]\\
				&=\underline{\epsilon}\circ\underline{S}^{-1}\left[s(\underline{\epsilon}(\underline{S}(h)_{(1)}))\underline{S}(h)_{(2)}\right] \\ &=\underline{\epsilon}\circ\underline{S}^{-1}\circ\underline{S}(h)=\underline{\epsilon}(h).
			\end{split}
		\end{equation*}
		where in order we used eqs.\eqref{eq:right_coprod_prop1}-\eqref{eq:right_coprod_prop4} of Lemma~\ref{lem:right_coprod}. 
		
		To conclude the proof we have to show that the action of $\underline{\epsilon}\circ\underline{S}'^{-1}\circ\underline{S}$ is unital and multiplicative. The first property comes from the fact that both compositions are of of unital maps. For the second proprrty,  for any $h,h'\in \mathcal{H}$ we compute 
		\begin{equation*}
			\begin{split}
				(hh')\triangleleft\underline{\epsilon}\circ\underline{S}'^{-1}\circ\underline{S}&=s\left[\epsilon\circ\underline{S}'^{-1}\circ\underline{S}(h_{(1)}h'_{(1)})\right]h_{(2)}h'_{(2)}\\
				&=s\circ\underline{\epsilon}\left[\underline{S}'^{-1}\circ\underline{S}(h_{(1)})\underline{S}'^{-1}\circ\underline{S}(h'_{(1)})\right]h_{(2)}h'_{(2)}\\
				&=s\circ\underline{\epsilon}\left[\underline{S}'^{-1}\circ\underline{S}(h_{(1)})\underline{S}'^{-1}\circ\underline{S}(h'_{(1)})\right]h_{(2)}h'_{(2)}\\
				&=s\circ\underline{\epsilon}\left[\underline{S}'^{-1}\circ\underline{S}(h_{(1)})t\circ\underline{\epsilon}\circ\underline{S}'^{-1}\circ\underline{S}(h'_{(1)})\right]h_{(2)}h'_{(2)}\\
				&=s\circ\underline{\epsilon}\left[\underline{S}'^{-1}\circ\underline{S}(h_{(1)}t\circ\underline{\epsilon}\circ\underline{S}'^{-1}\circ\underline{S}(h'_{(1)}))\right]h_{(2)}h'_{(2)}\\
				&=s\circ\underline{\epsilon}\left[\underline{S}'^{-1}\circ\underline{S}(h_{(1)})\right]h_{(2)}s\circ\underline{\epsilon}\left[\circ\underline{S}'^{-1}\circ\underline{S}(h'_{(1)})\right]h'_{(2)}\\
				&=(h\triangleleft\underline{\epsilon}\circ\underline{S}'^{-1}\circ\underline{S})(h'\triangleleft\underline{\epsilon}\circ\underline{S}'^{-1}\circ\underline{S}) .
			\end{split}
		\end{equation*}
		Here in the third line, we have used that $\underline{S'}^{-1}\circ\underline{S}$ is an algebra morphism, in the fourth eq.~\eqref{eq:counit_product}, in the fifth the combination of eqs.\eqref{eq:antipode_5} and \eqref{eq:antipode_6} and in the last one the Takeuchi product \eqref{eq:Takeuchi_prod}.
	\end{proof}
	
	This is a stronger version of Theorem~4.1 in \cite{bohm2019alternative} since we showed that the requirement in the Remark \ref{renos} is not needed.  
	
	\begin{remark}
		A \textit{Galois object} is a $H$-Hopf--Galois extension such that the subalgebra of coaction invariant elements is the ground field $B=\mathbb{K}$. In this case, the E-S is shown to be a Hopf algebra \cite{schau_bigal}. For a Hopf algebra, twists are convolution invertible characters and their action on the antipode provides an antipode in the sense of Hopf algebroid. In other words, starting from a Hopf algebra, using twists we get a full Hopf algebroid, a phenomenon which is central in \cite{connes1998hopf}.    
	\end{remark}

	\section{The Ehresmann--Schauenburg Hopf algebroid}
	We turn to the study of the theory of twists, as introduced previously, in the special case of the Erhesmann-Schauenburg (from now on E-S) bialgebroid associated to a principal $H$-comodule algebra $B\subseteq A$. 
	
	Thus we briefly recall the theory of the quantum gauge groupoid associated to a principal $H$-comodule algebra $B \subseteq A$. 
	The total algebra of the E-S bialgebroid is given by $\mathcal{C}(A,H)$ as defined in \ref{lem:equiv_of _coinvarinats}. Its $B$-coring structure is given by the following coproduct and counit: 
	\beq
	\underline{\Delta} :\mathcal{C}(A,H)\longrightarrow \mathcal{C}(A,H)\otimes_B \mathcal{C}(A,H), 
	\quad a\otimes \tilde{a}\longmapsto a_{(0)}\otimes \tensor{a}{_{(1)}^{\langle1\rangle}}
	\otimes_B\tensor{a}{_{(1)}^{\langle2\rangle}}\otimes  \tilde{a} \label{eq:coprod_bialgebroid} 
	\eeq
	\beq	
	\underline{\epsilon} :\mathcal{C}(A,H)\longrightarrow B , \quad
	a\otimes \tilde{a}\longmapsto a\tilde{a}  \, .
	\eeq
	
	\noindent
	The space $\mathcal{C}(A,H)$ is a subalgebra of $A\otimes A^{op}$, that is the product in $\mathcal{C}(A,H)$ is
	\begin{equation}
		(a\otimes \tilde{a})(c\otimes \tilde{c})=ac\otimes \tilde{c}\tilde{a},
	\end{equation}
	with $a\otimes \tilde{a}$ and $c\otimes \tilde{c}\in\mathcal{C}(A,H)$, and the $B^{e}$-ring structure is given 
	by the source and target maps
	\begin{align}
		\label{eq:source_bialgeboid}
		s:B\longrightarrow \mathcal{C}(A,H),\quad b\longmapsto b\otimes 1_A,  \\
		\label{eq:target_bialgeboid}
		t:B\longrightarrow \mathcal{C}(A,H),\quad b\longmapsto 1_A\otimes b.
	\end{align}
	All the compatibility conditions between the coring and ring structure are ensured for this case, so $\mathcal{C}(A,H)$ is a (left) $B$-bialgebroid.
	
	\begin{remark}
		Faithfully flatness of $A$ as a left $B$-module is sufficient to prove that the above set of data gives a bialgebroid over $B$. For the details see section $34$ in \cite{brz-wis}. An example of faithful flatness not being needed is in section $3$ in \cite{han2021gauge}.
	\end{remark}
	
	\subsection{The flip}
	We give a sufficient condition for the map $\mathrm{flip}_A:A\otimes A\to A\otimes A$, 
	$a \otimes \tilde{a} \to \tilde{a} \otimes a$, to be an antipode for the E-S bialgebroid of a Hopf--Galois extension $B\subseteq A$. Clearly $\mathrm{flip}_A$ is an anti-algebra morphism. It also satisfies, 
	for any $b\in B$, 
	\begin{equation*}
		\mathrm{flip}_A(t(b))=\mathrm{flip}(1_A\otimes b)=b\otimes1_A=s(b). 
	\end{equation*}
	Recall that $\mathcal{C}(A,H)$ is a sub-algebra of $A\otimes A^{op}$. 
	To prove that $\mathrm{flip}_A$ is actually an antipode it only remains to check eqs. \eqref{eq:antipode_2} and \eqref{eq:antipode_3}.
	
	\begin{prop}
		\label{prop:flip_antipode}
		Let $B\subseteq A$ be a principal $H$-comodule algebra and let $\mathcal{C}(A,H)$ be the associated E-S bialgebroid. If $\mathrm{flip}_A$ is a right $H$-comodule endomorphism 
		of $A\otimes A$ for the diagonal coaction $\rho^{\otimes}$, then is an antipode for $\mathcal{C}(A,H)$.
	\end{prop}
	\begin{proof}
		The coinvariance hypothesis on $\mathrm{flip}_A$ reads
		\begin{equation*}
			\rho^{\otimes}\circ\mathrm{flip}_A=(\mathrm{flip}_A\otimes\mathrm{id}_H)\circ\rho^{\otimes} .
		\end{equation*}
		Then, if $a\otimes \tilde{a}\in \mathcal{C}(A,H)=(A\otimes A)^{coH}$ so does $\mathrm{flip}_A(a\otimes \tilde{a})=\tilde{a}\otimes a$, that is the flip maps $\mathcal{C}(A,H)$ into itself. By the lemma \ref{lem:equiv_of _coinvarinats} we also have $\tilde{a}\otimes a\in\mathcal{L}_A$. For the left-side expression of eq.~\eqref{eq:antipode_2} we need
		\begin{align*}
			(a\otimes \tilde{a})_{(1)}\otimes_B(a\otimes \tilde{a})_{(2)}&=a_{(0)}\otimes\tensor{a}{_{(1)}^{\langle1\rangle}}\otimes_B\tensor{a}{_{(1)}^{\langle2\rangle}}\otimes \tilde{a},\\
			(\mathrm{flip}_A(a\otimes \tilde{a})_{(1)})_{(1')}\otimes_B(\mathrm{flip}_A(a\otimes \tilde{a})_{(1)})_{(2')}&=\tensor{a}{_{(1)}^{\langle1\rangle}_{(0)}}\otimes\tensor{a}{_{(1)}^{\langle1\rangle}_{(1)}^{\langle1\rangle}}\otimes_B\tensor{a}{_{(1)}^{\langle1\rangle}_{(1)}^{\langle2\rangle}}\otimes a_{(0)} .
		\end{align*}
		For the left hand side we thus get
		\begin{align*}
			\tensor{a}{_{(1)}^{\langle1\rangle}_{(0)}}&\tensor{a}{_{(1)}^{\langle2\rangle}}\otimes \tilde{a}\tensor{a}{_{(1)}^{\langle1\rangle}_{(1)}^{\langle1\rangle}}\otimes_B\tensor{a}{_{(1)}^{\langle1\rangle}_{(1)}^{\langle2\rangle}}\otimes a_{(0)}=\\
			=&\tensor{a}{_{(2)}^{\langle1\rangle}}\tensor{a}{_{(2)}^{\langle2\rangle}}\otimes \tilde{a}S(a_{(1)})^{\langle1\rangle}\otimes_BS(a_{(1)})^{\langle2\rangle}\otimes a_{(0)}\\
			=&1_A\otimes \tilde{a}S(a_{(1)})^{\langle1\rangle}\otimes_BS(a_{(1)})^{\langle2\rangle}\otimes a_{(0)}\\
			=&1_A\otimes \tilde{a}\tau(S(a_{(1)}))\otimes a_{(0)}\\
			=&1_A\otimes \tilde{a}_{(0)}\tau(\tilde{a}_{(1)})\otimes a\\
			=&1_A\otimes1_A\otimes_B\tilde{a}\otimes a=1_{\mathcal{C}(A,H)}\otimes\mathrm{flip}_A(a\otimes \tilde{a}) ;
		\end{align*}
		here in the first line we used eq.~\eqref{eq:transl_2}, in the second one the eq.~\eqref{eq:transl_3} and identity $a_{(1)}\epsilon(a_{(2)})=a_{(1)}$, in the third one just the definition of the translation map, in the fourth one the fact that $\tilde{a}\otimes a\in\mathcal{L}_A$ and finally in the fifth line eq.~\eqref{eq:transl_5}.
		
		Notice that $\mathrm{flip}_A^{-1}=\mathrm{flip}_A$, so to write down the expression on the left hand of eq.~\eqref{eq:antipode_3} in this case we need
		\begin{equation*}
			(\mathrm{flip}_A(a\otimes \tilde{a})_{(2)})_{(1')}\otimes_B(\mathrm{flip}_A(a\otimes \tilde{a})_{(2)})_{(2')}=\tilde{a}_{(0)}\otimes\tensor{\tilde{a}}{_{(1)}^{\langle1\rangle}}\otimes_B\tensor{\tilde{a}}{_{(1)}^{\langle2\rangle}}\otimes \tensor{a}{_{(1)}^{\langle2\rangle}}
		\end{equation*} 
		which becomes
		\begin{align*}
			\tilde{a}_{(0)}\otimes\tensor{\tilde{a}}{_{(1)}^{\langle1\rangle}}\otimes_B\tensor{\tilde{a}}{_{(1)}^{\langle2\rangle}}a_{(0)}\otimes\tensor{a}{_{(1)}^{\langle1\rangle}}\tensor{a}{_{(1)}^{\langle2\rangle}}&=\tilde{a}_{(0)}\otimes\tensor{\tilde{a}}{_{(1)}^{\langle1\rangle}}\otimes_B\tensor{\tilde{a}}{_{(1)}^{\langle2\rangle}}a_{(0)}\otimes\epsilon(a_{(1)})\\
			&=\tilde{a}_{(0)}\otimes\tau(\tilde{a}_{(1)})a\otimes1_A\\
			&=\tilde{a}\otimes a\otimes_B1_A\otimes 1_A=\mathrm{flip}_A(a\otimes \tilde{a})\otimes1_{\mathcal{C}(A,H)} .
		\end{align*}
		We used again eq.~\eqref{eq:transl_3}, the identity $a_{(0)}\epsilon(a_{(1)})=a$, and that $\tilde{a}\otimes a\in \mathcal{C}(A,H)$.
	\end{proof} 
	
	Using this result we have the following
	\begin{cor}
		\label{cor:antipode_comm}
		Let $B\subseteq A$ be a principal $H$-comodule algebra such that $H$ is commutative, then $\mathrm{flip}_A$ is an antipode for $\mathcal{C}(A,H)$.
	\end{cor}
	\begin{proof}
		In view of the previous proposition, it is enough to show that $\mathrm{flip}_A$ is a right $H$-comodule endomorphism of $(A\otimes A,\rho^{\otimes})$. This is easily checked since for any $a\otimes \tilde{a}\in A\otimes A$ using commutativity of $H$ one computes,
		\begin{align*}
			\rho^{\otimes}(\mathrm{flip}_A(a\otimes \tilde{a}))&=\tilde{a}_{(0)}\otimes a_{(0)}\otimes \tilde{a}_{(1)}a_{(1)}\\
			&=\tilde{a}_{(0)}\otimes a_{(0)}\otimes a_{(1)}\tilde{a}_{(1)}\\
			&=(\mathrm{flip}_A\otimes\mathrm{id}_H)(a_{(0)}\otimes \tilde{a}_{(0)}\otimes a_{(1)}\tilde{a}_{(1)})\\
			&=(\mathrm{flip}_A\otimes\mathrm{id}_H)(\rho^{\otimes}(a\otimes \tilde{a})) .
		\end{align*}
		This concludes the proof. 
	\end{proof}
	
	\subsection{Twists for the E-S bialgebroid}
	For the E-S bialgebroid associated to a Hopf--Galois extension $\mathcal{C}(A,H)$, 
	from the definition of $\mathcal{H}^*$, using the coproduct \eqref{eq:coprod_bialgebroid}
	and the source map in \eqref{eq:source_bialgeboid}, the right action \eqref{eq:right_H_*_action} read
	\beq
	(a\otimes \tilde{a}) \triangleleft\phi_* = \big(\phi_*(a_{(0)}\otimes\tensor{a}{_{(1)}^{\langle1\rangle}})\tensor{a}{_{(1)}^{\langle2\rangle}}\big)\otimes \tilde{a}. 
	\eeq
	Then for the conditions in \ref{defn:twists} we have $\phi_*(1_A\otimes 1_A) = 1_A \otimes 1_A$ 
	and
	\begin{multline}	 \label{eq:multiplicative_action_twists}
		\Big(\phi_*(a_{(0)}c_{(0)}\otimes\tensor{c}{_{(1)}^{\langle1\rangle}}\tensor{a}{_{(1)}^{\langle1\rangle}})\tensor{a}{_{(1)}^{\langle2\rangle}}\tensor{c}{_{(1)}^{\langle2\rangle}} \Big)\otimes \tilde{c}\tilde{a} \\ = \Big(\phi_*(a_{(0)}\otimes\tensor{a}{_{(1)}^{\langle1\rangle}})\tensor{a}{_{(1)}^{\langle2\rangle}}\phi_*(c_{(0)}\otimes\tensor{c}{_{(1)}^{\langle1\rangle}})\tensor{c}{_{(1)}^{\langle2\rangle}} \Big) \otimes \tilde{c}\tilde{a} , 
	\end{multline}
	for $a\otimes \tilde{a},c\otimes \tilde{c}\in\mathcal{C}(A,H)$.

	We next characterize the group of twists for the E-S bialgebroid. Let $\mathrm{Alg}^{H}(A)$ be the group of unital right $H$-comodule algebra automorphism of $A$, with product 
	\begin{equation}
		\label{eq:group_law_Alg}
		F\cdot G=G\circ F.
	\end{equation}
	\begin{remark}
		Attached to a principal $H$-comodule algebra $B\subseteq A$ there is also the group 
		$\tensor[]{\mathrm{Alg}}{_B}{^{\!\!\!\!\!H}}(A)$ of right $H$-comodule algebra morphisms $F:A\to A$ such that $F|_B=\mathrm{id}_B$ (they are also called vertical). The invertibility of the elements in $\tensor[]{\mathrm{Alg}}{^H_B}(A)$ follows from the Hopf-Galois condition and faithful flatness as proved in \cite{schneider1990principal}. This is the algebraic version of the commutative case that principal bundle morphisms are indeed isomorphisms \cite{husemoller1966fibre}. More generally, elements of 
		$F\in\tensor[]{\mathrm{Alg}}{^H}(A)$ when restricted to $B$ give an automorphism $F|_B:B\to B$.    
	\end{remark}

	\begin{prop}
		\label{prop:group_iso}
		Let $B\subseteq A$ be a principal $H$-comodule algebra, the formulas
		\begin{align}
			\label{eq:group_iso1}
			\phi_*^{F}(a\otimes \tilde{a}) &:=F(a)\tilde{a}, \qquad a\otimes \tilde{a}\in\mathcal{C}(A,H), \quad F\in\mathrm{Alg}^{H}(A),\\
			\label{eq:group_iso2}
			F_{\phi_*}(a) &:=\phi_*(a_{(0)}\otimes \tensor{a}{_{(1)}^{\langle1\rangle}})\tensor{a}{_{(1)}^{\langle2\rangle}},\qquad a\in A, \quad \phi_*\in\mathcal{T}^*,
		\end{align}
		provide a group isomorphism between the group $\mathrm{Alg}^{H}(A)$ and
		the group of twits $\mathcal{T}^*$ of the bialgebroid $\mathcal{C}(A,H)$. 
	\end{prop}
	\begin{proof}
		Let $F\in\mathrm{Alg}^{H}(A)$ and consider $\phi_*^{F}$ as in \eqref{eq:group_iso1}. 
		Its image lies in $B$ since $F$ is a $H$-comodule map:
		\begin{equation*}
			\rho(F(a)\tilde{a})=F(a_{(0)})\tilde{a}_{(0)}\otimes a_{(1)}\tilde{a}_{(1)}=F(a)\tilde{a}\otimes 1_H,
		\end{equation*}
		for any $a\otimes \tilde{a}\in\mathcal{C}(A,H)$. Moreover, one easily checks that it is a unital right $B$-module morphism. The right action \eqref{eq:right_H_*_action} here reads
		\begin{equation*}
			(a\otimes \tilde{a}) \triangleleft\phi_*^{F}=F(a)\otimes \tilde{a}
		\end{equation*}
		and with this we have
		\begin{equation*}
			\begin{split}
				(ac\otimes \tilde{c}\tilde{a}) \triangleleft\phi_*^{F}&=F(ac)\otimes \tilde{c}\tilde{a}\\
				&=(F(a)\otimes \tilde{a})(F(c)\otimes \tilde{c})\\
				&=\big( (a\otimes \tilde{a}) \triangleleft\phi_*^{F}\big) 
				\big((c\otimes \tilde{c}) \triangleleft\phi_*^{F})\big)
			\end{split}
		\end{equation*}
		for any $a\otimes \tilde{a}$ and $c\otimes \tilde{c}\in\mathcal{C}(A,H)$. 
		Now with a second $G\in\mathrm{Alg}^{H}(A)$, for any $a\otimes \tilde{a}\in\mathcal{C}(A,H)$, one finds
		\begin{equation*}
			\begin{split}
				\phi_*^{F\cdot G}(a\otimes \tilde{a})=(F\cdot G)(a)\tilde{a}=G(F(a))\tilde{a}=\phi_*^{G}(F(a)\otimes \tilde{a})=\phi_*^{F}\phi_*^{G}(a\otimes \tilde{a})
			\end{split}
		\end{equation*}
		and finally $\phi_*^{id_A}=\underline{\epsilon}$. All of this shows that $F\to\phi_*^F$ is a group morphism with $(\phi_*^F)^{-1}=\phi_*^{F^{-1}}$.
		
		Conversely, if $\phi_*\in\mathcal{T}^*$ is a twist the expression \eqref{eq:group_iso2} is well-defined due to the right $B$-linearity of $\phi_*$ and $a_{(0)}\otimes\tensor{a}{_{(1)}^{\langle1\rangle}}\otimes\tensor{a}{_{(1)}^{\langle2\rangle}}\in\mathcal{C}(A,H)\otimes A$. From \eqref{eq:twist_unital} $F_{\phi_*}$ is unital and from \eqref{eq:multiplicative_action_twists}, for any $a,c\in A$ one has
		\begin{equation*}
			\begin{split}
				\rho\circ F_{\phi_*}(a)=\Big(\phi_*(a_{(0)}\otimes\tensor{a}{_{(1)}^{\langle1\rangle}})\tensor{a}{_{(1)}^{\langle1\rangle}} \Big) \otimes a_{(2)}=(F_{\phi_*}\otimes\mathrm{id}_H)\circ\rho(a)
			\end{split}
		\end{equation*}
		\begin{equation*}
			\begin{split}
				F_{\phi_*}(ac)&=\phi_*\left(a_{(0)}c_{(0)}\otimes \tensor{c}{_{(1)}^{\langle1\rangle}}\tensor{a}{_{(1)}^{\langle1\rangle}}\right)\tensor{a}{_{(1)}^{\langle2\rangle}}\tensor{c}{_{(1)}^{\langle2\rangle}}\\
				&=\phi_*\left((a_{(0)}\otimes\tensor{a}{_{(1)}^{\langle1\rangle}})(c_{(0)}\otimes\tensor{c}{_{(1)}^{\langle1\rangle}})\right)\tensor{a}{_{(1)}^{\langle2\rangle}}\tensor{c}{_{(1)}^{\langle2\rangle}}\\
				&= \Big(\phi_*(a_{(0)}\otimes\tensor{a}{_{(1)}^{\langle1\rangle}})\tensor{a}{_{(1)}^{\langle2\rangle}}\phi_*(c_{(0)}\otimes\tensor{c}{_{(1)}^{\langle1\rangle}})\tensor{c}{_{(1)}^{\langle2\rangle}}\\
				&=F_{\phi_*}(a)F_{\phi_*}(c) .
			\end{split}
		\end{equation*}
		To prove these equations, we used 
		properties of the translation map in Proposition~\ref{translation_map_properties}.
		Thus the map $F_{\phi_*}$ is a right $H$-comodule algebra morphism.  Furthermore it is easy to check that $\phi_*\to F_{\phi_*}$ is a group morphism
		\begin{equation*}
			\begin{split}
				F_{\phi_*\psi_*}(a)&=\phi_*\psi_*(a_{(0)}\otimes\tensor{a}{_{(1)}^{\langle1\rangle}})\tensor{a}{_{(1)}^{\langle2\rangle}}\\
				&=\psi_*\left(\phi_*(a_{(0)}\otimes\tensor{a}{_{(1)}^{\langle1\rangle}})\tensor{a}{_{(1)}^{\langle2\rangle}}\otimes \tensor{a}{_{(2)}^{\langle1\rangle}}\right)\tensor{a}{_{(2)}^{\langle2\rangle}}\\
				&=F_{\psi_*}(\phi_*(a_{(0)}\otimes\tensor{a}{_{(1)}^{\langle1\rangle}})\tensor{a}{_{(1)}^{\langle2\rangle}})\\
				&=F_{\psi_*}(F_{\phi_*}(a))=F_{\phi_*}\cdot F_{\psi_*}(a)
			\end{split} 
		\end{equation*}
		for any $a\in A$ and $F_{\underline{\epsilon}}=\mathrm{id}_A$, thus $(F_{\phi_*})^{-1}=F_{\phi_*^{-1}}$. Finally one has
		$$
		\phi_*^{F_{\phi_*}}=\phi_* \, ,  \quad F_{\phi_{*}^{F}}=F ,
		$$
		thus finishing the proof. \end{proof}
	
	\begin{remark}
		The equations \eqref{eq:group_iso1} and \eqref{eq:group_iso2} realizing the isomorphism between $\mathcal{T}^*$ and $\mathrm{Alg}^H(A)$ are the same that give the isomorphism between the group of bisections of $\mathcal{B}(\mathcal{C}(A,H))$ and unital $H$-comodule algebra maps that preserve the base $B$ (vertical gauge transformations) $\tensor[]{\mathrm{Alg}}{_B}{^{\!\!\!\!\!H}}(A)$as proved in \cite{han2021gauge}. The invertibility of a $H$-comodule algebra map that is not vertical must be assumed in general.
	\end{remark}
	\section{The $U(1)$-extension ${\mathcal O}(\mathbb{C}P^{n-1}_q)\subseteq {\mathcal O}(S^{2n-1}_q)$}
	In this final section we study a special case of Hopf--Galois extension and associated Ehresmann--Schauenburg bialgebroid, the (non-commutative) principal $U(1)$-bundles $S^{2n-1}_q \to \mathbb{C}P^{n-1}_q$ over quantum projective spaces.
	
	\subsection{The Hopf--Galois structure}
	Let $q\in(0,1)$ and consider the $*$-algebra generated by elements $\{z_i$,$z^*_i \}_{i=1,\dots,n}$ subjected to relations:
	\begin{align}
		\label{eq:comm_sphere1}
		z_iz_j=qz_jz_i \quad \forall i<j,\quad z^*_iz_j=qz_jz^*_i \qquad \forall i\neq j\\
		\label{eq:comm_sphere2}
		[z^*_1, z_1]=0, \quad
		[z^*_k, z_k]=(1-q^2)\sum_{j=1}^{k-1}z_jz^*_j\qquad \forall 1<k\leq n
	\end{align}
	There is a sphere relation 
	\beq
	\label{eq:comm_sphere3}
	\sum_{j=1}^{n}z_jz^*_j=1 
	\eeq
	which, using \eqref{eq:comm_sphere2}, can be rewrite as
	\begin{equation}
		\label{eq:sphere_rel2}
		\sum_{j=1}^{n}q^{2(n-j)}z^*_jz_j=1 \, .
	\end{equation}
	This algebra is the coordinate algebra of  the \textit{quantum $(2n-1)$-dimensional sphere} 
	and is denoted ${\mathcal O}(S^{2n-1}_q)$. \noindent
	The entries of the projection $P_{ij}:=z^{*}_iz_j$ form a $*$-algebra which is the $q$-deformation of the coordinate algebra of the complex projective space that is denoted by ${\mathcal O}(\mathbb{C}P^{n-1}_q)$. The commutation relations among the generators $P_{ij}$ come from those 
	of ${\mathcal O}(S^{2n-1}_q)$ above.
	
	Let us denote by 
	${\mathcal O}(U(1)) =\mathbb{C}[t,t^{-1}]$ (Laurent polynomials) the Hopf $*$-algebra generated by a coordinate $t$ and its inverse with involution $t^*=t^{-1}$, and with coproduct, counit and antipode given by
	\begin{eqnarray}
		\label{eq:U(1)_Hopf_strc}
		\Delta(t^{\pm})=t^{\pm}\otimes t^{\pm}, & \epsilon(t^{\pm})=1, & S(t^{\pm})=t^{\mp} .
	\end{eqnarray}
	The algebra ${\mathcal O}(S^{2n-1}_q)$ is a right $\mathbb{C}[t,t^{-1}]$-comodule $*$-algebra if endowed with
	\begin{equation}
		\label{eq:U(1)_coaction}
		\rho:{\mathcal O}(S^{2n-1}_q)\longrightarrow {\mathcal O}(S^{2n-1}_q)\otimes\mathbb{C}[t,t^{-1}],\quad z_i\longmapsto z_i\otimes t,\quad z^*_i\longmapsto z^*_i\otimes t^{-1}
	\end{equation}
	One checks that the subalgebra of coaction invariants in ${\mathcal O}(S^{2n-1}_q)$ is ${\mathcal O}(\mathbb{C}P^{n-1}_q)$.
	Then ${\mathcal O}(\mathbb{C}P^{n-1}_q)\subseteq {\mathcal O}(S^{2n-1}_q)$ is a faithfully flat $\mathbb{C}[t,t^{-1}]$-Hopf--Galois extension. This is shown by observing that the canonical map, 
	\begin{equation*}
		\chi:{\mathcal O}(S^{2n-1}_q)\otimes_{{\mathcal O}(\mathbb{C}P^{n-1}_q)}{\mathcal O}(S^{2n-1}_q)\longrightarrow {\mathcal O}(S^{2n-1}_q)\otimes \mathbb{C}[t,t^{-1}]
	\end{equation*}  
	is surjective.  Indeed for the translation map one has:
	\begin{align}
		\label{eq:transl_map_proj}
		&\tau:\mathbb{C}[t,t^{-1}]\longrightarrow {\mathcal O}(S^{2n-1}_q)\otimes_{{\mathcal O}(\mathbb{C}P^{n-1}_q)}{\mathcal O}(S^{2n-1}_q)\nonumber\\ \tau(t)=\sum_{j=1}^{n}&q^{2(n-j)}z^*_j\otimes_{{\mathcal O}(\mathbb{C}P^{n-1}_q)}z_j,\quad \tau(t^{-1})=\sum_{j=1}^{n}z_j\otimes_{{\mathcal O}(\mathbb{C}P^{n-1}_q)}z^*_j
	\end{align} 
	for which one easily verifies, using \eqref{eq:comm_sphere3} and \eqref{eq:sphere_rel2}, 
	that $\mathrm{can}(\tau(t^{\pm}))=1\otimes t^{\pm}$ (which indeed suffices for surjectivity). Moreover the Hopf algebra $\mathbb{C}[t,t^{-1}]$ is cosemisimple and has a bijective antipode. Then from Theorem~1 in \cite{schneider1990principal} one has that ${\mathcal O}(\mathbb{C}P^{n-1}_q)\subseteq {\mathcal O}(S^{2n-1}_q)$ is a  principal $\mathbb{C}[t,t^{-1}]$-comodule algebra.

	\subsection{The $K$-theory of the base and the bialgebroid}
	The expression for the translation map in 
	eq.~\eqref{eq:transl_map_proj} can be written more compactly using the elements in the free module ${\mathcal O}(S^{2n-1}_q)^{n}\simeq {\mathcal O}(S^{2n-1}_q)\otimes\mathbb{C}^{n}$ given by
	\begin{equation}
		\label{eq:vector_projective_space1}
		v=\left(\begin{matrix}
			z^*_1, 
			z^*_2, 
			\dots, 
			z^*_{n-1},
			z^*_n
		\end{matrix}\right) ^t, \qquad	w=\left(\begin{matrix}
			q^{(n-1)}z_1, 
			q^{(n-2)}z_2, 
			\dots,
			q^1 z_{n-1}, 
			z_n 
		\end{matrix}\right)^t \, .
	\end{equation}
	They are easily seem to satisfy
	\begin{equation}
		\tau(t)=w^{\dagger}\dot{\otimes}_{{\mathcal O}(\mathbb{C}P^{n-1}_q)}w, \quad \tau(t^{-1})=v^{\dagger}\dot{\otimes}_{{\mathcal O}(\mathbb{C}P^{n-1}_q)}v \, .
	\end{equation}
	Here the symbol $\dot{\otimes}$ stands for the matrix multiplication composed with the tensor product.
	From \eqref{eq:comm_sphere3} and \eqref{eq:sphere_rel2} 
	both $v$ and $w$ are partial isometries 
	with respect to 
	the 
	restriction to ${\mathcal O}(S^{2n-1}_q)$ of the canonical hermitian product 
	on the free module ${\mathcal O}(S^{2n-1}_q)^n$ given by 
	$
	(\xi,\eta):=\sum_{j=1}^{n}\xi^*_j\eta_j ,
	$
	for $\xi=(\xi_j)$ and $\eta=(\eta_j)\in {\mathcal O}(S^{2n-1}_q)^n$. 
	
	As a consequence the two $n\times n$ matrices $P=vv^{\dagger}$ and $Q=ww^{\dagger}$ are projections with entries in the algebra ${\mathcal O}(\mathbb{C}P^{n-1}_q)$. They define two inequivalent classes \cite{d2015anti,hajac1999projective} in the K-theory $K_0({\mathcal O}(\mathbb{C}P^{n-1}_q))$. 
	The components of the matrices $P$ and $Q$ 
	are
	\begin{equation}
		\label{eq:Q_entries}
		P_{ij}:=z^{*}_iz_j \, , \qquad  Q_{ij}:=q^{(2n-i-j)}z_iz^*_j \, .
	\end{equation}

	\subsection{Full Hopf algebroid structures}
	Referring to the previous section, set $A={\mathcal O}(S^{2n-1}_q)$, $B={\mathcal O}(\mathbb{C}P^{n-1}_q)$ and $H=\mathbb{C}[t,t^{-1}]$. The two column vectors \eqref{eq:vector_projective_space1} give the generators of $\mathcal{C}(A,H)$ in the form of matrices $V:=v\dot{\otimes} v^{\dagger}$ and $W:=w\dot{\otimes} w^{\dagger}\nonumber$ whose entries are
	\beq
	\label{eq:matrices_V/W}
	V_{ij}=z^*_{i}\otimes z_j, \qquad W_{ij}=q^{(2n-i-j)}z_i\otimes z^*_j .
	\eeq
	Clearly, these are coaction invariant in $A\otimes A$ for the diagonal coaction. 
	In fact, any coaction invariant element in $A\otimes A$ can be written as a combination of $V_{ij}$'s 
	and $W_{ij}$'s. This follows from a simple argument: with the coaction in \eqref{eq:U(1)_coaction}, the $z^*_i$'s are of weight $-1$ and the $z_i$'s of weight $1$. Then any element of weight $0$ in $A\otimes A$, a coaction invariant, is a combination of elements of the form $z^*_i\otimes z_j$ and $z_i\otimes z^*_j$. The latter are proportional to $V_{ij}$ and $W_{ij}$.

	The commutations relations of the components of the matrices $V$ and $W$ in 
	\eqref{eq:matrices_V/W}
	can be derived from eqs.\eqref{eq:comm_sphere1} and \eqref{eq:comm_sphere2}.
	For the first we have
	\begin{align}
		\label{eq:comm_rel_V}
		V_{ik}V_{jk}&=q^{-1}V_{jk}V_{ik},\quad V_{ki}V_{kj}=q^{-1}V_{kj}V_{ki},\qquad \forall i<j \nonumber \\
		V_{ik}V_{jl}&=V_{jl}V_{ik},\quad V_{il}V_{jk}=q^{-2}V_{jk}V_{il},\qquad \forall i<j,l<k
	\end{align}
	while for the entries of $W$ one finds
	\begin{align}
		\label{eq:comm_rel_W}
		W_{ik}W_{jk}&=qW_{jk}W_{ik},\quad W_{ki}W_{kj}=qW_{kj}W_{ki},\quad \forall i<j \nonumber \\
		W_{ik}W_{jl}&=W_{jl}W_{ik},\quad W_{il}W_{jk}=q^2W_{jk}W_{il},\quad \forall i<j,l<k \, .
	\end{align}
	Following \eqref{eq:source_bialgeboid} and \eqref{eq:target_bialgeboid} the source and target maps are given by
	\begin{align}
		\label{eq:source_target_projective_sp}
		s: P_{ij}\longmapsto\sum_{k=1}^{n}q^{(j-k)}V_{ik}W_{jk},\quad Q_{ij}\longmapsto\sum_{k=1}^{n}q^{(k-j)}W_{ik}V_{jk} \nonumber \\
		t:P_{ij}\longmapsto\sum_{k=1}^{n}q^{(i-k)}V_{kj}W_{ki},\quad Q_{ij}\longmapsto\sum_{k=1}^{n}q^{(k-i)}W_{kj}V_{ki} \, . 
	\end{align}
	Thus the two embeddings of $B$ into $\mathcal{C}(A,H)$ are combinations of $V_{ij}$ and $W_{ij}$.

	For the coproduct and counit we find a similar formula as before
	\begin{align}
		\underline{\Delta}:V_{ij}\longmapsto\sum_{k=1}^{n}&V_{ik}\otimes_BV_{kj},\quad W_{ij}\longmapsto\sum_{k=1}^{n}W_{ik}\otimes_BW_{kj}\\
		&\underline{\epsilon}:V_{ij}\longmapsto P_{ij},\quad W_{ij}\longmapsto Q_{ij} \, .
	\end{align}
	Having set the notation for the bialgebroid we are ready for the next
	\begin{prop}
		\label{prop:antipode_proj}
		Define on the generators of $\mathcal{C}(A,H)$ the map
		\begin{equation}
			\label{eq:anitpode}
			\underline{S}: V_{ij}\longmapsto q^{(j-i)}W_{ji}, \quad W_{ij}\longmapsto q^{(i-j)}V_{ji} \, .
		\end{equation}
		Then, this is an antipode of $\mathcal{C}(A,H)$ with inverse given by
		\begin{equation}
			\label{eq:antipode_inv}
			\underline{S}^{-1}:V_{ij}\longmapsto q^{(i-j)}W_{ji}, \quad W_{ij} \longmapsto q^{(j-i)}V_{ji}
		\end{equation}
	\end{prop}
	\begin{proof}
		For the property of an antipode, firstly we have that $\underline{S}$ is an anti-algebra morphism: 
		for any $i<j$ one computes \begin{equation*}
			\begin{split}
				\underline{S}(V_{jk}) \underline{S}(V_{ik})&=q^{(k-j)}q^{(k-i)}W_{kj}W_{ki}\\
				&=q^{-1} (q^{(j-k)}W_{jk})(q^{(i-k)}W_{ik})\\
				&=q^{-1} \underline{S}(V_{ik})\underline{S}(V_{jk})
			\end{split}
		\end{equation*}
		when we used eq.~\eqref{eq:comm_rel_V}. The same goes for the generators $W_{ij}$.
		
		Looking at \eqref{eq:source_target_projective_sp} and the definition of $\underline{S}$ in this case, for eq.~\eqref{eq:antipode_1} one has
		\begin{equation*}
			\begin{split}
				\underline{S}(t(P_{ij}))&=\sum_{k=1}^{n}q^{(i-k)}\underline{S}(W_{ki})\underline{S}(V_{kj})\\
				&=\sum_{k}^{n}q^{(i-k)}q^{(k-i)}V_{ik}q^{(j-k)}W_{jk}\\
				&=\sum_{k=1}^{n}q^{(j-k)}V_{ik}W_{jk}=s(P_{ij})
			\end{split}
		\end{equation*}
		and with similar computations one finds the same for $Q_{ij}$.
		
		For the last property, \eqref{eq:antipode_2}, firstly we have
		\begin{equation*}
			\underline{\Delta}(V_{ij})=\sum_{k=1}^{n}V_{ik}\otimes_BV_{kj},\quad \underline{\Delta}(\underline{S}(V_{ik}))=q^{(k-i)}\sum_{l=1}^{n}W_{kl}\otimes_BW_{li} .
		\end{equation*}
		Then eq.~\eqref{eq:antipode_2} becomes
		\begin{equation*}
			\begin{split}
				\sum_{k,l=1}^{n}q^{(k-i)}W_{kl}V_{kj}\otimes_BW_{li}&=\sum_{k=1}^{n}\left(\sum_{l=1}^{n}q^{(k-j)}W_{kl}V_{kj}\right)\otimes_Bq^{(j-i)}W_{li}\\
				&=\sum_{l=1}^{n}t(Q_{jl})\otimes_Bq^{(j-i)}W_{li}\\
				&=1_{\mathcal{C}(A,H)}\otimes_Bq^{(j-i)}\sum_{l=1}^{n}s(Q_{jl})W_{li}\\
				&=1_{\mathcal{C}(A,H)}\otimes_Bq^{(j-i)}W_{ji}=1_{\mathcal{C}(A,H)} .\otimes_B\underline{S}(V_{ij})
			\end{split}
		\end{equation*}
		In the second line we used the tensor product over $B$ in \eqref{eq:tensor_B}, while in the third one the relation
		$
		\sum_{l=1}^{n}s(Q_{jl})W_{li}=W_{ji}
		$
		which is just  \eqref{eq:counital_bialgebroid}. The same goes for the generators $W_{ij}$. 
		
		We prove eq.~\eqref{eq:antipode_3} now starting with $W_{ij}$. Firstly, 
		\begin{equation*}
			\underline{\Delta}(W_{ij})=\sum_{k=1}^{n}W_{ik}\otimes_BW_{kj},\quad \underline{\Delta}(\underline{S}^{-1}(W_{kj}))=q^{(j-k)}\sum_{l=1}^{n}V_{jl}\otimes_BV_{lk} .
		\end{equation*}
		Thus, one gets
		\begin{equation*}
			\begin{split}
				\sum_{k,l=1}^{n}q^{(j-k)}V_{jl}\otimes_BV_{lk}W_{ik}&=q^{(j-i)}\sum_{l=1}^{n}V_{jl}\otimes_B\sum_{k=1}^{n}q^{(i-k)}V_{lk}W_{ik}\\
				&=q^{(j-i)}\sum_{l=1}^{n}V_{jl}\otimes_Bs(P_{li})\\
				&=q^{(j-i)}\sum_{l=1}^{n}t(P_{li})V_{jl}\otimes_B 1_{\mathcal{C}(A,H)}\\
				&=q^{(j-i)}V_{ji}\otimes_B1_{\mathcal{C}(A,H)}=\underline{S}^{-1}(W_{ij})\otimes_B1_{\mathcal{C}(A,H)}.
			\end{split}
		\end{equation*}
		Between the second and third line, we used again the structure \eqref{eq:tensor_B} and 
		\eqref{eq:counital_bialgebroid} in the form
		$
		\sum_{l=1}^{n}t(P_{li})V_{jl}=V_{ji} .
		$
		Similar computations hold for the generators $V_{ij}$.
		
		Being $V_{ij}$ and $W_{ij}$ generators of $\mathcal{C}(A,H)$ we can conclude that the map $\underline{S}$ 
		of eq.~\eqref{eq:anitpode} is an invertible antipode of the Ehresmann--Schauenburg bialgebroid associated to the $\mathbb{C}[t,t^{-1}]$-Hopf--Galois extension ${\mathcal O}(\mathbb{C}P^{n-1}_q)\subseteq {\mathcal O}(S^{2n-1}_q)$, bialgebroid that is indeed a Hopf algebroid.
	\end{proof}
	
	We are thus in the situation of having two antipodes on the bialgebroid associated with the extension ${\mathcal O}(\mathbb{C}P^{n-1}_q)\subseteq {\mathcal O}(S^{2n-1}_q)$. 
	One is just the map \eqref{eq:anitpode} as we have shown above. For the second one, 
	being the Hopf algebra $\mathbb{C}[t,t^{-1}]$ commutative, 
	we saw in Proposition~\ref{prop:flip_antipode} and \ref{cor:antipode_comm} that the map $\mathrm{flip}_{{\mathcal O}(S^{2n-1}_q)}$ is an antipode, its action on the generators in
	\eqref{eq:matrices_V/W} is given by
	\begin{equation*}
		\mathrm{flip}_{{\mathcal O}(S^{2n-1}_q)}(V_{ij})=q^{(i+j-2n)}W_{ji},\quad \mathrm{flip}_{{\mathcal O}(S^{2n-1}_q)}(W_{ij})=q^{(2n-i-j)}V_{ji}.
	\end{equation*}
	Using the theory developed in subsection 2.2, the two are related by the twist
	\begin{equation}
		\psi_*(V_{ij})=q^{2(i-n)}P_{ij},\quad \psi_*(W_{ij})=q^{2(n-j)}Q_{ij},
	\end{equation}
	and thus by Theorem~\ref{thm:Bohm_thm} we get
	\begin{equation}
		\mathrm{flip}_{{\mathcal O}(S^{2n-1}_q)}(\cdot)=\underline{S}(\cdot\triangleleft\psi_*).
	\end{equation} 
	
	We retrieve the whole group of twists from Proposition~\ref{prop:group_iso}
	\begin{prop}
		\label{prop:twist_proj_sp}
		For the ${\mathcal O}(U(1))$-Hopf--Galois extension ${\mathcal O}(\mathbb{C}P^{n-1}_q)\subseteq {\mathcal O}(S^{2n-1}_q)$ for the group of twists one has $\mathcal{T}^*\simeq(\mathbb{C}^*)^{n}$.
	\end{prop}
	\begin{proof}
		We first work out the group $\mathrm{Alg}^{{\mathcal O}(U(1)) } ({\mathcal O}(S^{2n-1}_q))$ and then use the 
		isomorphism of Proposition~\ref{prop:group_iso} before. One easily checks that any right $H$-comodule automorphism of $A$ has the form
		\begin{eqnarray}
			F(z_i)=X_iz_i, & & F(z_i^*)=Y_iz^*_i, \nonumber
		\end{eqnarray}
		where $X_i, Y_i\in {\mathcal O}(\mathbb{C}P^{n-1}_q)$. The invertible ones are such that $X_i, Y_i\in\mathbb{C}^*$ (non-zero complex numbers), since they exhauast  all the invertible elements in ${\mathcal O}(\mathbb{C}P^{n.1}_q)$. Moreover we require $F(1_A)=1_A$. The above expression is clearly an algebra morphism if $Y_i=X^{-1}_i$ (so that it preserves the sphere relations). 
		Using eq.~\eqref{eq:group_iso1} the action of twists on the generators 
		of $\mathcal{C}({\mathcal O}(S^{2n-1}_q),{\mathcal O}(U(1))$ is
		$$
		\phi_*^F(V_{ij})=X^{-1}_iP_{ij} , \qquad \phi_*^F(W_{ij})=X_iQ_{ij} .
		$$
		This finishes the proof.
	\end{proof}

	\bigskip
	\noindent{\bf Acknowledgements.}
	
	LD acknowledges partial support from the "National Group for Mathematical Physics" (GNFM--INdAM).
	LD and JZ acknowledge partial support from the EU Horizon Project "Graph Algebras" Grant agreement ID: 101086394. GL acknowledges partial support from INFN, Iniziativa Specifica GAST, 
	from the ``National Group for Algebraic and Geometric Structures, and their
	Applications'' (GNSAGA--INdAM) and from 
	the PNRR PE National Quantum Science and Technology Institute (PE0000023).

	\bibliography{bib}{}
	\bibliographystyle{plain} 
\end{document}